\documentclass[a4paper,12pt]{article}
\usepackage[all]{xy} 
\usepackage{amssymb}
\usepackage{epsfig}
\usepackage{amsfonts}
\usepackage{amsmath}
\usepackage{euscript}
\usepackage{amscd}
\usepackage{amsthm}
\DeclareMathAlphabet{\mathpzc}{OT1}{pzc}{m}{it}

\newtheorem{thm}{Theorem}[section]
\newtheorem{lem}[thm]{Lemma}
\newtheorem{prop}[thm]{Proposition} 
\newtheorem{cor}[thm]{Corollary}
\newtheorem{rem}[thm]{Remark}
\newtheorem{ex}[thm]{Example}

\newcommand{\m}{\mathpzc{m}}
\newcommand{\p}{\mathpzc{p}}

\newcommand{\bQ}{\mathbb Q}

\newcommand{\bC}{\mathbb C}

\newcommand{\hgt}{\operatorname{ht}}

\newcommand{\noi}{\noindent}

\newtheorem{Q}[thm]{Question}
\numberwithin{equation}{section}

\newcommand{\q}{\mathfrak{q}}
\newcommand{\n}{\mathfrak{n}}

\begin{document}


\title{On finite generation of Noetherian algebras
over two-dimensional regular local rings}
\author{ Amartya Kumar Dutta$^\star$, Neena Gupta$^{\star\star}$ and Nobuharu Onoda$^\dagger$\footnote{{Present address: Fukui Study Center, The Open University of
Japan, AOSSA 7F, Teyose 1-4-1, Fukui 910-0858, Japan}}\\
{\small{\it $^\star$, $^{\star\star}$ Stat-Math Unit, Indian Statistical Institute,}}\\
{\small{\it 203 B.T. Road, Kolkata 700 108, India.}}\\
{\small{\it e-mail : amartya.28@gmail.com and neenag@isical.ac.in}}\\
{\small{\it $^\dagger$Department of Mathematics, University of Fukui,}}\\
{\small{\it Bunkyo 3-9-1, Fukui 910-8507, Japan}}\\
{\small{\it e-mail :  onoda@u-fukui.ac.jp}}
}
\maketitle
\date{}


\begin{abstract}
Let $R$ be a complete regular local ring with an algebraically closed residue
field and let $A$ be a Noetherian $R$-subalgebra of the polynomial
ring $R[X]$. It has been shown in \cite{DO2} that if $\dim R=1$, then 
$A$ is necessarily finitely generated over $R$. In this paper, we 
give necessary and  sufficient conditions for $A$ 
to be finitely generated over $R$ when $\dim R=2$ and 
present an example of a Noetherian normal non-finitely 
generated $R$-subalgebra of $R[X]$ 
over $R= \bC[[u, v]]$.

\smallskip



\end{abstract}


%
%

\section{Introduction}
Let $R$ be a Noetherian ring and $A$ 
a Noetherian $R$-subalgebra of $R[X]$, where $R[X]$ is the polynomial
ring in one indeterminate $X$ over $R$.
Then $A$ need not be a finitely generated $R$-algebra, in general. 
In fact, by an example of Eakin \cite[p. 79]{Ea},  
even when $R$ is the polynomial ring $\bC[t]$, 
there exist non-finitely generated
Noetherian rings $A$ satisfying $R \subseteq A\subseteq R[X]$. 
However, it has been shown in \cite{DO2}
that when $R$ is a complete discrete valuation ring 
with an algebraically closed residue field
(e.g., when $R= \bC[[t]]$) then any Noetherian $R$-subalgebra $A$ of $R[X]$
must be finitely generated.
A precise version of the result is quoted below (\cite[Theorem 4.2]{DO2}).

\begin{thm}\label{thmB}
Let $(R,\pi)$ be a complete discrete valuation ring 
with residue field $k$ and field of fractions $K$.
Suppose that the algebraic closure $\bar{k}$ of $k$ 
is a finite extension of $k$. Let $A$ be a Noetherian domain containing
$R$ such that $A[\pi^{-1}]$ is a finitely generated $K$-algebra and
${\rm tr.deg}_R A = 1$. Then $A$ is finitely generated over $R$. 
In particular, if $A$ is a Noetherian $R$-subalgebra of $R[X]$, then $A$ is 
finitely generated over $R$.
\end{thm}

Over a complete discrete valuation ring, the following result 
(\cite[Proposition 3.4]{DO2}) relates the finite 
generation of a Noetherian normal $R$-subalgebra of $R[X]$
with the transcendence degree of certain fibres. 

\begin{thm}\label{prop34}
Let $(R,\pi)$ be a complete discrete valuation ring 
with residue field $k$ and field of fractions $K$.
Let $A$ be a Krull domain such that $R\subseteq A$,  
$A[\pi^{-1}]$ is a finitely generated $K$-algebra and
${\rm tr.deg}_R A = 1$. Then $A$ is finitely generated over $R$ if 
${\rm tr.deg}_{k}\, A/P >0$
for each associated prime ideal $P$ of $\pi A$. 
\end{thm}

In this paper, we explore two-dimensional analogues of the above results.
Note that the integral domain $A$ is flat over the discrete valuation ring $R$ in Theorem \ref{thmB} and 
this leads us to the following question:

\begin{Q}\label{q1}
{\rm 
Let $R$ be a complete regular local ring of 
dimension two with algebraically closed residue field. Let $A$  
be a Noetherian  $R$-subalgebra of $R[X]$. Is $A$ 
finitely generated over $R$, say when $A$ is normal and flat over $R$?
}
\end{Q}

Recall that if the ring $A (\neq R)$ in the above question 
is factorial, then $A \cong_R R[X]$
by a result of Abhyankar-Eakin-Heinzer \cite[Theorem 4.1]{AEH}. 

Example \ref{ex1}, the main example of this paper, provides 
a counter example to the above question.
However, our main theorem (Theorem \ref{Main}) shows that 
under an additional fibre condition  on $A$ 
(similar to the criterion in Theorem \ref{prop34}),
we do have a two-dimensional analogue of Theorem \ref{thmB}. 
We quote a consequence of our main result (cf. Theorem \ref{Main} and 
Corollary \ref{rx}):

\medskip

\noi
{\bf Theorem I.}
{\it Let $(R, \m)$ be a complete two-dimensional regular local ring
with residue field $k$. Suppose that the algebraic closure
$\bar{k}$ of $k$ is a finite extension of $k$. 
Let $A$ be a Noetherian domain such that $A$ is a flat $R$-algebra
with ${\rm tr.deg}_R A = 1$. 
Suppose that there exists $\pi \in \m$ such that $A[\pi^{-1}]$ 
is a finitely generated $R[\pi^{-1}]$-algebra
and ${\rm tr.deg}_{R/(P\cap R)}\, A/P >0$  
for each $P\in {\rm Ass}_A(A/\pi A)$.  
Then $A$ is finitely generated over $R$.
}

\medskip

Regarding Question \ref{q1},
it has been shown in \cite[Lemma 3.3]{DO2} that 
over a one-dimensional Noetherian domain $R$, 
any Krull domain which is an $R$-subalgebra of $R[X]$ is Noetherian.
We give an example (Example \ref{ex2}) to show that 
this result does not hold when
${\rm dim}~R>1$, not even when $R$ is a complete regular local 
domain with an algebraically closed residue field.

We now give a layout of the paper. In Section 2, we recall 
a few results which we shall use to prove our statements.
In Section 3, as a step towards Theorem I, we first prove 
generalisations of Theorems \ref{thmB} and \ref{prop34} 
(Propositions \ref{4.2} and \ref{3.4}). 
Then we consider the case where the base ring $R$ is a two-dimensional
Noetherian local domain, and establish a few criteria, including Theorem I,
for Noetherian property 
and finite generation of $R$-subalgebras of 
a finitely generated $R$-algebra. 
We also give a sufficient condition for the ring $A$ of Theorem I to be 
Noetherian under some fibre conditions on $A$ (Theorem \ref{noeth}).
In Section 4, we demonstrate our examples. These examples are based on
the methodology of Lemmas \ref{l1}--\ref{l3}. In Appendix, 
we establish a condition for finite generation of
algebras over excellent rings. This result was earlier established by 
the third author in \cite{Ot},
over fields and the proof is essentially the same.

\section{Preliminaries}
Throughout the paper $R$ will denote a commutative ring with unity. The
notation $A= R^{[n]}$ will denote that $A$ is 
a polynomial ring in $n$ variables over $R$.  For an element $c$ in $A$, the 
notation $A_c$ will denote the ring $T^{-1}A$, where $T$ is the multiplicatively closed set $\{c^n~|~n \ge 0\}$.

\smallskip

\noindent
{\bf Definition}. A Noetherian ring $R$ is said to be a Nagata ring 
(or a pseudo-geometric ring) if, 
for every prime ideal $\p$ of $R$ and for every finite algebraic extension 
field $L$ of the field of fractions $k(\p)$ of $R/\p$, 
the integral closure of $R/\p$ in $L$ is a finite module over $R/\p$.

Any Noetherian complete local ring is a Nagata ring (\cite[p. 234, Corollary 2]{Ma})
and any finitely generated algebra over a Nagata ring is a Nagata ring (\cite[p. 240, Theorem 72]{Ma}).

\smallskip

We first recall the following version of dimension 
inequality (cf. \cite[Theorem 2]{C} 
\begin{thm}\label{C}
Let $R$ be a Noetherian integral domain and $B$ an integral domain
containing $R$. Let $P$ be a prime ideal of $B$ and $\p= P \cap  R$.
Then 
\begin{equation}\label{de}
\hgt P + {\rm tr.deg}_{R/\p} \, B/P \le \hgt \p + {\rm tr.deg}_R \, B.
\end{equation}
\end{thm}

Let the notation and assumptions be the same as in Theorem \ref{C}. 
Then we say that $P$  satisfies the dimension equality relative to $R$
if the equality holds in (\ref{de}), and 
we say that the dimension formula holds between $R$ and $B$ if every
prime ideal $P$ in $B$ satisfies the dimension equality relative to $R$.
It is known that if $R$ is universally catenary
and $B$ is finitely generated over $R$, then the dimension formula
holds between $R$ and $B$ (cf. \cite[Theorem 15.6]{M}). For later use
we note the following.

\begin{lem}\label{DE}
Let $R\subseteq B$ be integral domains such that $R$ is Noetherian. 
Let $P$ be a prime ideal in $B$
and let $\p=P\cap R$. Suppose that $\hgt \p=1$. 
Then ${\rm tr.deg}_{R/\p} B/P \leq{\rm tr.deg}_RB$, where the equality holds 
if and only if $\hgt P=1$ and 
$P$ satisfies the dimension equality relative to $R$.
In particular, the equality holds if $\hgt P=1$, $R$ is universally catenary,
and $B$ is finitely generated over $R$.
\end{lem}
\begin{proof}
Since $\hgt \p=1$, we have 
\[
{\rm tr.deg}_{R/\p} B/P \leq 1-\hgt P +{\rm tr.deg}_RB \leq
{\rm tr.deg}_RB
\]
by Theorem \ref{C}. From this it follows that 
${\rm tr.deg}_{R/\p} B/P \leq {\rm tr.deg}_RB$, and 
the equality holds if and only if
$\hgt P=1$ and $P$ satisfies the dimension equality relative to $R$. 
\end{proof}

For convenience, we now quote a few other known results which will be needed in our arguments.
We first state an easy lemma.

\begin{lem}\label{easy}
Let $B \subseteq A$ be integral domains. Suppose that there exists a nonzero
element $t$ in $B$ such that $B[t^{-1}] = A[t^{-1}]$ and $tA\cap B=tB$. Then $B = A$.
In particular, if $t$ is a nonzero prime element in $B$, then (by letting
$A=B[t^{-1}]\cap B_{(tB)}$), we have $B=B[t^{-1}]\cap B_{(tB)}$.
\end{lem}

The following result, giving a criterion for an integral domain to be Noetherian, 
is proved in \cite[Lemma 2.8]{DO2}. 

\begin{lem}\label{cohen}
Let $D$ be an integral domain. Suppose that there exists a nonzero element $t$ in $D$ such that
\begin{enumerate}
\item[{\rm (I)}] $D[t^{-1}]$ is a Noetherian ring.
\item[{\rm (II)}] $tD$ is a maximal ideal of $D$.
\item[{\rm (III)}] $\hgt (tD)=1$
(or, equivalently, $\displaystyle{\bigcap_{n \ge 1}}\,t^nD=(0)$).
\end{enumerate}
Then $D$ is a Noetherian ring.
\end{lem}

For ready reference, we state below a comaximality criterion for a ring to be Noetherian, or an algebra to be
finitely generated, which can be proved easily.

\begin{lem}\label{noethl}
Let $a$ and $b$ be two regular elements of a ring $B$ such that $(a,b)B=B$. 
Then the following statements hold.
\begin{enumerate}
\item[{\rm (1)}] If $B_a$ and $B_b$ are Noetherian, then $B$ is Noetherian.   
\item[{\rm (2)}] If $B$ is an $R$-algebra such that $B_a$ and $B_b$ are finitely generated $R$-algebras,
then $B$ is a finitely generated $R$-algebra.
\end{enumerate}
\end{lem}

We state below another elementary result on finite generation
(\cite[p.\,201]{N2}).

\begin{lem}\label{integral}
Let $R$ be a Noetherian domain and $B$ an $R$-subalgebra of a finitely
generated $R$-algebra $C$. 
If $C$ is integral over $B$, then $B$ is finitely generated over $R$.
\end{lem}

For a proof of the following result on finite generation, see \cite[2.1]{G} or \cite[Proposition 2.11]{O}.
\begin{prop}\label{giral}
Let $R$ be a Noetherian domain and $A$ a subalgebra of a finitely generated $R$-algebra.
Then there exists a nonzero element $f$ in $A$ such that $A[f^{-1}]$ is a finitely generated $R$-algebra.
\end{prop}

We recall below the local-global result  \cite[Theorem 2.20]{O} which reduces the question 
of finite generation of a subalgebra of a polynomial algebra
to the local situation. Recall that an integral domain $C$ containing $R$ is said to be a locality
(or essentially of finite type) over $R$ if there exists a finitely generated $R$-algebra $B$ and a prime ideal $Q$ of
$B$ such that $C=B_Q$.

\begin{thm}\label{onoda}
Let $R$ be a Noetherian domain and $B$ an integral domain containing $R$ 
such that there exists a nonzero $f \in B$ for which 
$B[f^{-1}]$ is a finitely generated $R$-algebra. Then the following statements hold.
\begin{enumerate}
\item[\rm{(1)}] If $B_M$ is a locality over $R$ for every maximal ideal $M$ of $B$, then $B$ is
a finitely generated $R$-algebra. 
\item[\rm{(2)}] If $B_{\m}$ is a finitely generated $R_{\m}$-algebra for each
maximal ideal $\m$ of $R$, then $B$ is a finitely generated $R$-algebra.
\end{enumerate}
\end{thm}

We now state a criterion for a torsion-free module over a two-dimensional regular local ring
to be flat. For the lack of a ready reference, we give below a proof.

\begin{lem}\label{cflat}
Let $R$ be a two-dimensional regular local ring and $\{\pi_1, \pi_2 \}$ a regular system of
parameters in $R$.
Let $M$ 
be a torsion-free $R$-module such that $\pi_2$ is  $(M/\pi_1M)$-regular.  
Then $M$ is flat over $R$. 
\end{lem}
\begin{proof}
 By \cite[Lemma 4.1]{BG2}, it suffices to show that for every prime ideal $P$ of $R$,
the natural map 
\begin{equation}\label{injective}
\phi:  P \otimes_R M \longrightarrow M
\end{equation}
 is injective.  

Fix a prime ideal $P$ of $R$. 
Suppose $\hgt (P)=0$. Then $P= 0$ as $R$ is an integral domain and 
hence the natural map $\phi$ in (\ref{injective}) is trivially injective.  

Next, suppose that $\hgt (P)=1$. Since $R$ is a regular local ring and hence a 
UFD, $P$ is a principal ideal. 
Since $M$ is torsion-free, it then follows that the natural map $\phi$ in (\ref{injective})
is injective. 

Finally, we consider the case
$\hgt (P)=2$, i.e., $P=(\pi_1, \pi_2)R$, the unique maximal ideal of $R$. 
Then, any element $\xi$ in $P \otimes_R M$ 
can be expressed as 
\begin{equation}\label{xi}
\xi= \pi_1 \otimes m_1+ \pi_2 \otimes m_2
\end{equation}
for some $m_1,\ m_2 \in M$.
Suppose that $\phi(\xi)=0$, i.e., 
\begin{equation}\label{regular}
\pi_1 m_1 + \pi_2 m_2 = 0. 
\end{equation}
Since  $\pi_2$ is ($M/\pi_1M$)-regular, we have $m_2=\pi_1 m$ for some $m \in M$.
Since $M$ is torsion-free, from (\ref{regular}), we have $m_1 = -\pi_2 m$. 
Substituting in (\ref{xi}), we see that $\xi=0$, i.e., 
$\phi$ is injective. 
\end{proof}

\section{Main Results}

In this section we shall prove our main theorem
and record a few auxiliary observations. 

For the proof of our main theorem, we first record two 
results on finite generation over one-dimensional Nagata domains, 
which are generalisations of Theorems \ref{prop34}
and \ref{thmB}, respectively. 
\begin{prop}\label{3.4}
Let $R$ be a one-dimensional Nagata domain, and
let $B$ be a Krull domain such that 
$R\subseteq B$ with ${\rm tr.deg}_R B = 1$. Suppose that there
exists $\pi\in R$ such that $B[\pi^{-1}]$ 
is a finitely generated $R[\pi^{-1}]$-algebra.
Then the following statements are equivalent.
\begin{enumerate}
\item[{\rm(i)}] $B$ is a finitely generated $R$-algebra. 
\item[{\rm(ii)}] ${\rm tr.deg}_{R/P\cap R} B/P >0$ for every minimal prime ideal 
$P$ of $\pi B$.
\end{enumerate}
\end{prop}
\begin{proof}
(i) $\Rightarrow$ (ii).
Let $P$ be a minimal prime ideal of $\pi B$. Then $\hgt P =1$ as $B$ is a Krull domain and
$\hgt (P \cap R) =1$ as ${\rm ~dim~}(R)=1$ and $P \cap R \ne 0$. By \cite[p.\,255, Corollary 2]{M},
$R$ is universally catenary.  Hence,  by the dimension formula
Lemma \ref{DE}, we have ${\rm tr.deg}_{R/P\cap R} B/P =1$.

\smallskip

(ii) $\Rightarrow$ (i).
Let $\bar{R}$ denote the normalisation of $R$. 
Then $\bar{R}\subseteq B$, and $\bar{R}$ is a finite $R$-module because $R$ is 
a Nagata domain. Hence it is enough to show that 
$B$ is a finitely generated $\bar{R}$-algebra. 

If $P$ is a minimal prime ideal of $\pi B$, 
$\p=P\cap R$ and $\p'=P\cap \bar{R}$,  
then $\bar{R}/\p'$ is integral over $R/\p$,
so that ${\rm tr.deg}_{\bar{R}/\p'}B/P={\rm tr.deg}_{R/\p}B/P>0$.
Thus, replacing $R$ by $\bar{R}$, we assume that $R$ is a Dedekind Nagata domain.

Since $B[\pi^{-1}]$ is a finitely generated $R[\pi^{-1}]$-algebra, 
by Theorem \ref{onoda}(2), it suffices  to show that 
$B_{\m} (= B \otimes_{R} {R}_{\m})$ is a 
finitely generated ${R}_{\m}$-algebra
for every maximal ideal $\m$ of ${R}$ with $\pi\in\m$. 

Fix a maximal ideal $\m$ of $R$ containing $\pi$ and set $k:=R/\m=R_{\m}/\m R_{\m}$.
Now note that for any $P\in{\rm Spec} B$, if $PB_{\m}$ is
a minimal prime ideal of $\pi B_{\m}$,
then $P$ is a minimal prime ideal of $\pi B$. Moreover, $PB_{\m} \cap R_{\m} =\m R_{\m}$ 
and, from our hypothesis, it follows that ${\rm tr.deg}_kB_\m/PB_\m>0$. 
Thus, replacing $R$ by $R_{\m}$, we may further assume that 
$R$ is a discrete valuation ring which is also a Nagata ring. 

The desired result now follows from \cite[Proposition 3.4]{DO2}.
\end{proof}
\begin{prop}\label{4.2}
Let $(R, \m)$ be a complete one-dimensional Noetherian local domain
with residue field $k$ and field of fractions $K$.
Suppose that the algebraic closure $\bar{k}$ of $k$ is a 
finite extension of $k$. 
Let $B$ be a Noetherian integral domain such that 
$R\subseteq B$, ${\rm tr.deg}_R B \le 1$, and  
$B\otimes_R K$ is a finitely generated $K$-algebra. 
Then $B$ is a finitely generated $R$-algebra.  
\end{prop}
\begin{proof}
We note that $\dim B \le 2$. Indeed, if $P$ is a prime ideal of $B$, 
then setting $\p := P \cap R$, we have $\hgt \p \le 1$, so that,
by Theorem \ref{C},
\[
\hgt P \le \hgt \p +{\rm tr.deg}_R B -{\rm tr.deg}_{R/\p} B/P\le 2,
\]
as claimed. 
Let $\bar{R}$ be the normalisation of $R$ and 
$\bar{B}$ the normalisation of $B$. 
Then $\bar{R}$ is a complete discrete valuation ring
which is a finite $R$-module (cf. \cite[p.\,263]{M}). 
Let $k'$ denote the residue field of $\bar{R}$. 
Then $k'$ is algebraic over $k$ and 
hence $\bar{k}$ is a finite algebraic extension of $k'$. 
Since $\dim B \le 2$, $\bar{B}$ is a Noetherian domain (cf. \cite[Theorems 33.2, 33.12]{N}). 
Note that $K=R[\pi^{-1}]$ for any $\pi (\ne 0) \in \m$. 
Since $\bar{R} \subseteq \bar{B}$, ${\rm tr.deg}_{\bar{R}} \bar{B} \le 1$ and 
$\bar{B} \otimes_{\bar{R}} K$ is a finitely generated $K$-algebra, 
it follows from \cite[Theorem 4.2]{DO2} that 
$\bar{B}$ is a finitely generated $\bar{R}$-algebra.
Hence $\bar{B}$ is a finitely generated $R$-algebra, 
because $\bar{R}$ is a finite $R$-module. 
Since $\bar{B}$ is integral over $B$, $B$  
is a finitely generated $R$-algebra by Lemma \ref{integral}. 
\end{proof}

We also need the following technical result in 
the proof of our main theorem.

\begin{prop}\label{avfin}
Let $(R, \m)$ be a two-dimensional Nagata local domain with
residue field $k$, and let $A$ be an integral domain such that 
$R\subseteq A$ with ${\rm tr.deg}_R A = 1$. For $P \in {\rm Spec}~A$, let $\p$ denote $P\cap R$.
Set $\Delta:=\{P\in {\rm Ass}_A(A/\pi A)~|~\hgt \p=1\}$. Suppose that $A$ satisfies the
following hypotheses.
\begin{enumerate}
\item[\rm (I)] $A$ is either a Krull domain or a Noetherian domain.
\item[\rm (II)] There exists a nonzero element $\pi \in \m$ such that $A[\pi^{-1}]$ is finitely generated over $R[\pi^{-1}]$.
\item[\rm (III)] ${\rm tr.deg}_{R/\p} A/P >0$ for every $P\in \Delta$.
\end{enumerate}  
Then the following statements hold.
\begin{enumerate}
\item[\rm (1)]
$A[f^{-1}]$ is finitely generated over $R[f^{-1}]$ for every
nonzero $f \in \m$.
\item[\rm (2)] Every $P \in \Delta$ has the following properties.  
\begin{enumerate}
\item[\rm(a)] ${\rm tr.deg}_{R/\p}A/P=1$.
\item[\rm(b)] $P$ satisfies the dimension equality relative to $R$. 
\item[\rm(c)] $A/P \otimes_{R/\p} k(\p)$ is finitely generated over $k(\p)$, 
where $k(\p)=R_\p/\p R_\p$ is the field of fractions of $R/\p$.
\end{enumerate}
\end{enumerate}
\end{prop}
\begin{proof}
(1) Fix $f (\ne 0) \in \m$. Set $S:= R[f^{-1}]$ and $C:= A[f^{-1}]$.
Note that $S$ is a one-dimensional Nagata domain.
Also note that $A[\pi^{-1}]$ is a Nagata domain because $A[\pi^{-1}]$
is finitely generated over $R[\pi^{-1}]$ and $R[\pi^{-1}]$ is 
a Nagata domain.

\smallskip

Now, if $f \in \sqrt{\pi A}$, then $f^n=\pi a$ for some $n>0$ and
$a\in A$, so that $C=A[f^{-n}]=A[\pi^{-1}, a^{-1}]$ is 
finitely generated over $S$ and we are through. So we assume that $f \notin \sqrt{\pi A}$,
i.e., $\pi C$ is a proper ideal of $C$.  

\smallskip

\noindent
{\bf Case (i).} First, we consider the case where $A$ is a Krull domain.  
In this case, $C$ is a Krull domain such that
$C[\pi^{-1}]$ is finitely generated over $S[\pi^{-1}]$.
Let $Q$ be a minimal prime ideal of $\pi C$. Set
$P:=Q\cap A$ and $\q:=P\cap S=Q\cap S$. Then $P$ is
a minimal prime ideal of $\pi A$ and $\p=P\cap R=\q\cap R$.
Note that $\hgt \q=1$ as $\dim S=1$, and hence, 
$\hgt \p=1$. Therefore, ${\rm tr.deg}_{S/\q}C/Q$ ($={\rm tr.deg}_{R/\p}A/P$) $>0$, 
by our hypothesis. Hence, $C$ is finitely generated over $S$ by Proposition \ref{3.4}.

\smallskip

\noindent
{\bf Case (ii).} Next, we consider the case where $A$ is a Noetherian domain.
Let $D$ denote the normalisation of $A$. Then $D$ is a Krull domain 
(cf. \cite[Theorem 33.10]{N}). We shall verify that $D$ too satisfies all the hypotheses given for $A$.
 
Since $D[\pi^{-1}]$ is the normalisation
of the Nagata domain $A[\pi^{-1}]$, it follows that $D[\pi^{-1}]$ is
a finite $A[\pi^{-1}]$-module, and hence $D[\pi^{-1}]$
is finitely generated over $R[\pi^{-1}]$.

Let $Q$ be a minimal prime ideal of $\pi D$ such that
$\hgt (Q\cap R)=1$. Let $P= Q \cap A$. Then $\p=P\cap R=Q\cap R$. 
By \cite[Theorem 33.11]{N}, $P$ is an associated prime ideal of $\pi A$
 so that ${\rm tr.deg}_{R/\p}A/P>0$
by our hypothesis. Since $D/Q$ is integral over 
$A/P$, it follows that 
${\rm tr.deg}_{R/\p} D/Q>0$. 

Thus, by Case (i), $D[f^{-1}]$ is finitely generated over $R[f^{-1}]$.
Since $D[f^{-1}]$ is integral over $A[f^{-1}]$, we conclude, by Lemma \ref{integral},  
that $A[f^{-1}]$ is finitely generated over $R[f^{-1}]$.

\medskip

(2) Fix $P \in \Delta$. 

\smallskip

\noindent
(2a) Since $\hgt\p=1$, we have ${\rm tr.deg}_{R/\p}A/P\le {\rm tr.deg}_RA$
by Lemma \ref{DE}. But, by our hypotheses, ${\rm tr.deg}_RA=1$ and ${\rm tr.deg}_{R/\p}A/P \ge 1$.
Thus, ${\rm tr.deg}_{R/\p}A/P={\rm tr.deg}_RA=1$.

\smallskip

\noindent
(2b) Since the equality ${\rm tr.deg}_{R/\p}A/P={\rm tr.deg}_RA$ holds, $P$
satisfies the dimension equality relative to $R$ again by Lemma \ref{DE}. 

\smallskip

\noindent
(2c) Set $R':=R/\p$ and $B:=A/P$.
 Then $R'$ is a one-dimensional local domain. 
Hence, for any $f (\ne 0) \in \m\setminus\p$, we have
$R'[f^{-1}]=k(\p)$, so that $B\otimes_{R'}k(\p)=A[f^{-1}]/PA[f^{-1}]$.
Thus, the assertion is an immediate consequence of (1).
\end{proof}

We have the following consequence for the case $R$ is complete and the residue field of $R$
is an algebraically closed field or a real closed field.

\begin{cor}\label{avfincor}
Let $(R, \m)$ be a two-dimensional complete local domain with
residue field $k$ such that $[\bar{k}:k]<\infty$, where $\bar{k}$ denotes the
algebraic closure of $k$. Let $A$ be an integral domain satisfying all the hypotheses
of Proposition \ref{avfin}. Let $\Delta$ be as in Proposition \ref{avfin}.
Suppose that $A/P$ is Noetherian for some $P \in \Delta$.
Then $A/P$ is finitely generated over $R/\p$, where $\p=P \cap R$ as in Proposition \ref{avfin}.
\end{cor}

\begin{proof}
Since $R/\p$ is a one-dimensional complete local domain with
residue field $k$, the assertion follows from Proposition \ref{avfin}(2) and
Proposition \ref{4.2}.
\end{proof}

We also need the following easy lemma.

\begin{lem}\label{lemma_add}
Let $R$ be a Noetherian domain and $A$ a Krull domain containing $R$. 
Let $\pi$ be a nonzero element in $A$ with $\pi A\ne A$.
Then the following assertions hold.
\begin{enumerate}
\item[\rm(1)]
If $A/P$ is finitely generated over $R$ for every
$P\in{\rm Ass}_A(A/\pi A)$, then $A/\pi A$ is 
finitely generated over $R$.
\item[\rm(2)]
If $A/P$ is a finite $R$-module for every $P\in{\rm Ass}_A(A/\pi A)$,
then $A/\pi A$ is a finite $R$-module.
\end{enumerate}
\end{lem}
\begin{proof}
Let ${\rm Ass}_A(A/\pi A)=\{P_1,\dots,P_n\}$, and let
$\pi A= {P_1}^{(e_1)}\cap{P_2}^{(e_2)}\cap \cdots \cap {P_n}^{(e_n)}$
be the primary decomposition of $\pi A$. Then
\begin{equation}\label{deco}
A/{\pi A} \hookrightarrow B:=A/{P_1}^{(e_1)} \times A/{P_2}^{(e_2)} 
\times \cdots \times A/{P_n}^{(e_n)}
\end{equation}
is a finite integral extension. Hence if $B$ is finitely generated over
$R$, then so is $A/\pi A$ by Lemma \ref{integral}, and if $B$ is a finite $R$-module,
then so is $A/\pi A$ as $R$ is Noetherian.

Now, let $P\in {\rm Ass}_A(A/\pi A)$ and let $e$ be a positive
integer. Note that in both the cases (1) and (2), $A/P$ is
Noetherian, and hence so is $A/P^{(e)}$ by \cite[Theorem 12.7]{M}.
Note also that, setting $C:=A/P^{(e)}$, we have $C/\sqrt{(0)}=
A/P$. Therefore if $A/P$ is finitely generated over $R$
(resp. a finite $R$-module), then $A/P^{(e)}$ is also
finitely generated over $R$ (resp. a finite $R$-module)
(cf. \cite[Lemma 4.1]{DO2}). Thus the above ring $B$ in (\ref{deco}) 
is finitely generated over $R$ for the case (1), and is a finite $R$-module
for the case (2). This completes the proof.
\end{proof}

We are now ready to prove our main theorem. 

\begin{thm}\label{Main}
Let $(R, \m)$ be a complete two-dimensional Noetherian local domain
whose residue field $k$ satisfies the condition 
$[\bar{k}:k]<\infty$, where $\bar{k}$ is the algebraic closure of $k$.
Let $A$ be a Krull domain such that $R\subseteq A$ with ${\rm tr.deg}_R A = 1$.
Suppose $A$ satisfies the following hypotheses.
\begin{enumerate}
\item[\rm (I)] There exists a nonzero element $\pi\in\m$ such that 
$A[\pi^{-1}]$ is finitely generated over $R[\pi^{-1}]$.
\item[\rm (II)] $\hgt (P \cap R)=1$ for every $P\in{\rm Ass}_A(A/\pi A)$.
\end{enumerate} 
Then the following conditions are equivalent.
\begin{enumerate}
\item [\rm (i)] $A$ is finitely generated over $R$.
\item [\rm (ii)] $A$ is Noetherian and ${\rm tr.deg}_{R/\p}\, A/P >0$ 
for each $P\in{\rm Ass}_A(A/\pi A)$.
\item[\rm (iii)] $A/\pi A$ is Noetherian and ${\rm tr.deg}_{R/\p}\, A/P >0$ 
for each $P\in{\rm Ass}_A(A/\pi A)$.
\item [\rm (iv)] $A/P$ is Noetherian and ${\rm tr.deg}_{R/\p}\, A/P >0$ 
for each $P\in{\rm Ass}_A(A/\pi A)$. 
\end{enumerate}
\end{thm}
\begin{proof}
(i) $\Rightarrow$ (ii).
$A$ is Noetherian by Hilbert Basis Theorem.  Let $P\in{\rm Ass}_A(A/\pi A)$
and set $\p:=P \cap R$. Then $\hgt P=1$ because $A$ is Krull. 
Therefore, since the complete local ring $R$ is universally
catenary (cf. \cite[Theorem 29.4(ii)]{M}) and $A$ is finitely generated over $R$, 
by Lemma \ref{DE}, we have ${\rm tr.deg}_{R/\p} \, A/P ={\rm tr.deg}_RA=1$. 

\smallskip

(ii) $\Rightarrow$ (iii) and (iii) $\Rightarrow$ (iv) are trivial.

\smallskip

(iv) $\Rightarrow$ (i). 
Note that the Noetherian complete local ring $R$ is an excellent local ring (\cite[p. 260, 34.B]{Ma}).
We now verify that the element $\pi$ in the Krull domain $A$ satisfies all the hypotheses of
Proposition \ref{toya} in Section 5.

By hypothesis (I), $A[{\pi}^{-1}]$ is a finitely generated $R$-algebra. 

The hypothesis (II) and the conditions in (iv) show, by Corollary \ref{avfincor}, that
$A/P$ is finitely generated over $R/(P \cap R)$ for every $P\in{\rm Ass}_A(A/\pi A)$. 
Therefore, by Lemma \ref{lemma_add}, $A/\pi A$ is finitely generated over $R$. 

Finally, let $P$ be a minimal prime ideal of $\pi A$.
Then $\hgt P=1$ as $A$ is a Krull domain.
Also, by Proposition \ref{avfin}(2b), $P$ satisfies the dimension equality relative to $R$. 

Therefore, applying Proposition \ref{toya}, we conclude that $A$ is 
finitely generated over $R$.
\end{proof}

\begin{cor}\label{Mainc}
Let $(R,\m)$ and $k$ be as in Theorem \ref{Main} and let $A$ be
a Noetherian domain such that $R\subseteq A$ with ${\rm tr.deg}_R A = 1$.
Suppose $A$ satisfies the following hypotheses.
\begin{enumerate}
\item[\rm (I)] There exists a nonzero element $\pi\in\m$ such that 
$A[\pi^{-1}]$ is finitely generated over $R[\pi^{-1}]$.
\item[\rm (II)] $\hgt (P \cap R)=1$
 and ${\rm tr.deg}_{R/(P\cap R)}\, A/P >0$ for every $P\in{\rm Ass}_A(A/\pi A)$.
\end{enumerate}
Then $A$ is finitely generated over $R$.
\end{cor}
\begin{proof}
Let $D$ be the normalisation of $A$. $D$ is a Krull domain by \cite[Theorem 33.10]{N}.
Since $D$ is integral over $A$, by
Lemma \ref{integral}, it suffices to show that $D$ is finitely generated over $R$.
We verify that the Krull domain
$D$ satisfies all the general hypotheses and condition (iv) 
of Theorem \ref{Main}.
Let $Q\in{\rm Ass}_D(D/\pi D)$, $P=Q \cap A$ and $\p=Q\cap R=P \cap R$.
We will show:
\begin{enumerate}
\item[\rm(a)] $\hgt \p =1$.
\item[\rm(b)] For every $f \in \m \setminus \p$,  $D[f^{-1}]$ is finitely generated over $R[f^{-1}]$.
\item[\rm(c)] $D/Q$ is Noetherian.
\item[\rm(d)] ${\rm tr.deg}_{R/\p}D/Q>0$.
\end{enumerate}  

(a) By \cite[Theorem 33.10]{N}, $P\in{\rm Ass}_A(A/\pi A)$ and hence $\hgt \p =1$ by hypothesis (II).

(b) Fix $f \in \m \setminus \p$.  By Proposition \ref{avfin} (1), $A[f^{-1}]$
is finitely generated over the Nagata ring $R[f^{-1}]$ and hence $A[f^{-1}]$ is also a Nagata ring.
Thus $D[f^{-1}]$, being the normalisation of the Nagata ring $A[f^{-1}]$, is a finite $A[f^{-1}]$-module.
Hence, $D[f^{-1}]$ is finitely generated over $R[f^{-1}]$.

(c) By Corollary \ref{avfincor}, $A/P$ is finitely generated over the Nagata ring $R$ and hence
$A/P$ is a Nagata ring. Since $D[f^{-1}$ is a finite $A[f^{-1}]$-module, it follows 
that the field of fractions of $D/Q$ is a finite
extension of the field of fractions of $A/P$. Since $D/Q$ is integral over the Nagata ring $A/P$,
it then follows that $D/Q$
is a finite $A/P$-module; in particular $D/Q$ is Noetherian.

(d) Since $D/Q$ is a finite $A/P$-module, ${\rm tr.deg}_{R/\p}D/Q = {\rm tr.deg}_{R/(P\cap R)}\, A/P >0$
by hypothesis (II).

\smallskip

Thus, $D$ is finitely generated over $R$ by Theorem \ref{Main} and hence $A$ is finitely generated over $R$
by Lemma \ref{integral}.
\end{proof}

In the next result, we will see that
the hypothesis ``$\hgt(P \cap R)=1$ for $P\in{\rm Ass}_A(A/\pi A)$''
in the above result can be replaced by the condition ``$A$ is $R$-flat''.

\begin{cor}\label{rx}
Let $(R, \m)$ be a complete two-dimensional regular local domain
with residue field $k$. Suppose that the algebraic closure
$\bar{k}$ of $k$ is a finite extension of $k$. 
Let $A$ be a Noetherian domain such that $A$ is a flat $R$-algebra
with ${\rm tr.deg}_R A = 1$. 
Suppose that there exists $\pi \in \m$ such that $A[\pi^{-1}]$ 
is a finitely generated $R[\pi^{-1}]$-algebra
and ${\rm tr.deg}_{R/(P\cap R)}\, A/P >0$  
for each $P\in {\rm Ass}_A(A/\pi A)$.  
Then $A$ is finitely generated over $R$.
\end{cor}

\begin{proof}
Let $\pi (\ne 0) \in\m$, $P \in {\rm Ass}_A(A/\pi A)$ and $\p = P \cap R$.
Since $A$ is $R$-flat, we have ${\rm depth}R_\p\le {\rm depth}A_P=1$. Hence 
${\rm depth}R_\p=1$, which implies that $\hgt\p=1$, because $R$ is regular.
The result now follows from Corollary \ref{Mainc}.
\end{proof}

\begin{rem}\label{equals1}
{\em The respective proofs show that the condition ``${\rm tr.deg}_{R/(P\cap R)}\, A/P >0$ for every $P\in{\rm Ass}_A(A/\pi A)$''
occurring in the above results may be replaced by the equivalent condition 
``${\rm tr.deg}_{R/(P\cap R)}\, A/P =1$ for every $P\in{\rm Ass}_A(A/\pi A)$''.
}
\end{rem}

With the same notation and assumptions as in Theorem \ref{Main}, 
we shall now give sufficient conditions for the ring 
$A$ to be Noetherian in the case where
${\rm tr.deg}_{R/\p} A/P=0$ for every $P\in{\rm Ass}_A(A/\pi A)$ (Theorem \ref{noeth}).
We begin by recording an auxiliary result. 

\begin{lem}\label{lem2}
Let $(S, \n)$ be a complete one-dimensional Noetherian local domain
whose residue field $k$ satisfies the condition that
$[\bar{k}:k]<\infty$. Suppose that $B$ is a Noetherian domain 
such that $S\subseteq B$, ${\rm tr.deg}_S B= 0$ and $\n B \neq B$. 
Then $B$ is a finite $S$-module. 
\end{lem}

\begin{proof}
Since $B$ is a Noetherian domain, $B$ is separated for the $\n$-adic
topology. Thus, by \cite[Theorem 8.4]{M}, it suffices to show that 
$B/\n B$ is a finite $k$-module.

Let $\n B = Q_1 \cap Q_2\cap \cdots\cap Q_n$ be an irredundant 
primary decomposition of $\n B$ in $B$ and
let $P_i=\sqrt{Q_i}$ for $i=1,\dots,n$. 
Then $P_i \cap S = \n$ for each $i$, because $\dim S=1$.
Since ${\rm tr.deg}_S B = 0$, it then follows from
the dimension inequality that
\[
\hgt P_i + {\rm tr.deg}_k B/P_i \le \hgt \n + {\rm tr.deg}_S B=1
\]
for each $i$. From this we have that $\hgt P_i = 1$ and 
${\rm tr.deg}_k B/P_i = 0$, so that $B/P_i$ is a finite $k$-module
for each $i$, because $[\bar{k}:k]<\infty$.
Since $P_i^r\subseteq Q_i$ for some $r>0$, it thus follows that
each $B/Q_i$ is a finite $k$-module. Note that
\[
B/\n B \hookrightarrow B/Q_1\times B/Q_2 \times \cdots\times B/Q_n
\]
is a finite extension. Therefore $B/\n B$ is a finite $k$-module,
as desired. 
\end{proof}

\begin{lem}\label{field} Let $(S, \n)$ be a one-dimensional local domain and $B$ 
an integral domain containing $S$ such that $\n B=B$ and  ${\rm tr.deg}_S B =0$.
Then $B$ is a field.
\end{lem}

\begin{proof}
Let $P$ be an arbitrary prime ideal of $B$.
Since $(0)$ and $\n$ are the only prime ideals of the one-dimensional
local domain $(S, \n)$ and $\n B=B$, we have $P\cap S=(0)$. 

Let $K$ denote the field of fractions of $S$. Since $P\cap S=(0)$,
we have $K\subseteq B_P$.  Since ${\rm tr.deg}_{S} B = 0$, it then follows that
$B_P$ is algebraic over the field $K$. Hence $B_P$ is a field, which
implies that $P=(0)$. Thus, $B$ is a field.
\end{proof}

\begin{cor}\label{fieldcor}
Let $(R, \m)$ be a two-dimensional local domain and $A$ a Krull domain such that $R \subseteq A$,
$\m A=A$ and ${\rm tr.deg}_R A >0$. Let $\pi$ be a nonzero element in $\m$ and $P$ a minimal prime ideal
of $\pi A$. Suppose that ${\rm tr.deg}_{R/(P \cap R)} A/P=0$. Then $A/P$ is a field.
\end{cor}

\begin{proof}
Set $\p:=P\cap R$, $S:=R/\p$, $\n:=\m/\p$ and $B:=A/P$. Then $S \hookrightarrow B$ and identifying $S$ with its image in $B$,
we may assume that $S \subseteq B$.
Since $\m A=A$, we have $\n B=B$. The result now follows from Lemma \ref{field}.  
\end{proof}

\begin{prop}\label{lNoeth}
Let $(R, \m)$ be a two-dimensional complete Noetherian 
local domain whose residue field $k$ satisfies 
the condition that $[\bar{k}:k]<\infty$,
and let $\pi$ be a nonzero prime element of $\m$.
Let $D$ be an integral domain containing $R$ such that 
$D/\pi D$ is a Noetherian domain, $\bigcap_{n \ge 0}\pi^n D = 0$ 
and $\pi D \cap R= \pi R$. 
Then the following assertions hold.
\begin{enumerate}
\item [\rm (1)] If ${\rm tr.deg}_{R/ \pi R} D/\pi D = 0$ and $\m D \neq D$,
then $D$ is a finite $R$-module.
\item [\rm (2)] If ${\rm tr.deg}_R D >0$, then either 
${\rm tr.deg}_{R/ \pi R} D/\pi D > 0$ or $\m D = D$.
\item [\rm (3)] If ${\rm tr.deg}_{R/ \pi R} D/\pi D = 0$ and
$\m D=D$, then $D/\pi D$ is a field. 
In addition, if $D[\pi^{-1}]$ is Noetherian, 
then $D$ is Noetherian.
\end{enumerate}
\end{prop}

\begin{proof}
Set $S=R/\pi R$, $\n=\m/\pi R$ and $B=D/\pi D$.
Then, by assumptions, $(S,\n)$ is a complete one-dimensional Noetherian local
domain with residue field $k$ satisfying the condition that 
$[\bar{k}:k]<\infty$, $B$ is a Noetherian
domain and $S \hookrightarrow B$ so that,  identifying $S$ with its image in $B$,
we may assume that $S \subseteq B$.

\smallskip

(1) Note that $\n B\ne B$ because $\m D\ne D$. 
Since ${\rm tr.deg}_{S} B = 0$, it follows from
Lemma \ref{lem2} that $B$ is a finite $S$-module. 
On the other hand, since $R$ is $\m$-adically complete
and $\pi\in \m$, $R$ is $\pi$-adically complete, too. 
Since $\bigcap_{n \ge 0} \pi^n D= 0$, it now follows from
\cite[Theorem 8.4]{M} that $D$ is a finite $R$-module.

\smallskip

(2) The assertion follows from (1).

\smallskip

(3) The assertions follow from Corollary \ref{fieldcor} and Lemma \ref{cohen} respectively. 
\end{proof}

We now state our result giving criteria for the ring $A$ to be Noetherian
when $A/P$ is algebraic over $R/(P\cap R)$ for every minimal prime $P$ of $\pi A$.

\begin{thm}\label{noeth}
Let $(R, \m)$ be a two-dimensional complete Noetherian 
local domain whose residue field $k$ satisfies 
the condition that $[\bar{k}:k]<\infty$.
Let $A$ be a Krull domain such that $R\subseteq A$ with ${\rm tr.deg}_R A >0$.
Suppose that the following conditions hold.
\begin{enumerate}
\item[\rm(I)] There exists $\pi\in\m$ such that $A[\pi^{-1}]$ is Noetherian.
\item[\rm(II)] For every minimal prime ideal $P$ of $\pi A$, 
$\hgt (P \cap R)=1$ and ${\rm tr.deg}_{R/(P \cap R)} A/P=0$. 
\end{enumerate}
Then the following statements are equivalent.
\begin{enumerate}
\item[\rm(i)] $A$ is Noetherian.
\item[\rm(ii)] $A/\pi A$ is Noetherian.
\item[\rm(iii)] $\m A=A$.
\end{enumerate}
\end{thm}
\begin{proof}
(i) $\Rightarrow$ (ii) is obvious. 

\smallskip

(ii) $\Rightarrow$ (iii).
Suppose on the contrary that $\m A\ne A$, and 
let $M$ be a maximal ideal of $A$ such that $\m A \subseteq M$.
Set $A' = A_M$. Then $A'$ is a Krull local domain satisfying
$\m A'\ne A'$. Let $P'\in{\rm Ass}_{A'}(A'/\pi A')$, $P=P'\cap A$
and $\p= P \cap R$. Then $P\in{\rm Ass}_A(A/\pi A)$ and $A'/P'$ is
a localisation of $A/P$, and hence 
${\rm tr.deg}_{R/\p} A'/P' =0$ and $A'/P'$ is Noetherian by
our hypothesis. 
Note that $R/\p$ is a complete one-dimensional local
domain with residue field $k$ satisfying the condition that 
$[\bar{k}:k]<\infty$. Since ${\rm tr.deg}_{R/\p} (A'/P') = 0$
and $\m(A'/P')\ne A'/P'$, it now follows from
Lemma \ref{lem2} that $A'/P'$ is a finite $R/\p$-module. 
Therefore, by Lemma \ref{lemma_add}, we know that
$A'/\pi A'$ is a finite $R$-module, which means $A'/\pi A'$ is
a finite $R/\pi R$-module because $\pi R\subseteq \pi A'\cap R$.
Note that $\cap_{n \ge 0}\pi^nA'= (0)$ because
$A'$ is a Krull domain. Note also that $R$ is complete with 
respect to  $\pi$-adic topology because $\pi\in\m$.
It thus follows from \cite[Theorem 8.4]{M} that $A'$ is 
a finite $R$-module, which contradicts the condition that
${\rm tr.deg}_R A'>0$. Therefore $\m A=A$.

\smallskip

(iii) $\Rightarrow$ (i).
To prove that $A$ is Noetherian, it suffices to prove, by the Mori-Nishimura Theorem
(\cite[Theorem 12.7]{M}) that $A/N$ is Noetherian for every prime ideal $N$ in $A$ of height one.

Fix a prime ideal $N$ in $A$ of height one.
If $\pi\in N$, then $N$ is a minimal prime ideal of $\pi A$ and hence, $A/N$ is a field by Corollary \ref{fieldcor}; 
in particular, $A/N$ is Noetherian. 

Now we consider the case $\pi\notin N$. 
Since $A$ is a Krull domain and $P$
is a maximal ideal for every $P\in{\rm Ass}_A(A/\pi A)$ by Corollary \ref{fieldcor}, it then follows
that $N$ and $\pi A$ are comaximal ideals.
Therefore, $A/N=A[\pi^{-1}]/N[\pi^{-1}]$,
so that $A/N$ is Noetherian, since $A[\pi^{-1}]$ is Noetherian.

Thus, $A$ is Noetherian.
\end{proof}

\begin{rem}
{\em 
(1) The ring $A$ in Example \ref{ex1} shows the necessity of 
the hypothesis that
``${\rm tr.deg}_{R/\p R} A/P>0$ for each $P \in {\rm Ass}_A(A/\pi A)$'' in Theorem \ref{Main}. 

\smallskip

(2) The ring $D$ in Example \ref{non-noeth} shows 
the necessity of the hypothesis ``$R$ is complete''
in Proposition \ref{lNoeth}.

\smallskip

(3) The hypothesis $[\bar{k}:k]<\infty$, occurring in most of the results in this section, is equivalent to the condition 
$[\bar{k}:k]\le 2$; this condition is satisfied if and only if $k$ is either
an algebraically closed field or a real closed field (cf. \cite[Remark 4.3(2)]{DO2})
}
\end{rem}

\section{Examples}

In this section we shall give an example (Example \ref{ex1}) of a Noetherian normal
non-finitely generated subalgebra of the polynomial ring $R[X]$ over a two-dimensional 
complete regular local ring $R$. 

We shall first give methods (Lemmas \ref{l1}--\ref{l3}) for constructing 
Noetherian normal $R$-subalgebras of $R[X]$, when 
$R$ is a Noetherian normal domain with field of fractions $K$.
Lemma \ref{l1} considers a Krull subring $D$ of $K[X]$ with certain properties,
Lemma \ref{l2} examines the ring $A:= D\cap R[X]$ and 
Lemma \ref{l3} gives a sufficient criterion for $D$ and $A$ to be Noetherian.
These results are generalisations of Lemma 5.4 in \cite{DO2}. 

Throughout this section we denote by $X$ an indeterminate over $R$.

\begin{lem}\label{l1}
Let $R$ be a Noetherian normal domain, and 
let $\pi$ be a nonzero prime element of $R$. Let 
$D$ be an integral domain containing $R$ such that 
\begin{enumerate}
\item[\rm(I)] $D[1/\pi] = R[1/\pi][X]$;
\item[\rm(II)] $\pi D$ is a prime ideal and $\pi D \cap R =\pi R$;
\item[\rm(III)] $D_{(\pi D)}$ is a discrete valuation ring.
\end{enumerate}
Then the following statements hold.
\begin{enumerate}
\item[\rm(1)] $D$ is a Krull domain.
\item[\rm(2)] If $R$ a UFD then $D$ is a UFD.
\item[\rm(3)] If $p$ is a prime element of $R$, then $p$ remains a prime element in $D$ 
and either $pR=\pi R$ or $pD_{(\pi D)}=D_{(\pi D)}$.
\end{enumerate}
\end{lem}
\begin{proof}
(1) By Lemma \ref{easy}, $D = D[1/\pi] \cap D_{(\pi D)}$.
Hence $D$ is Krull, because both $D[1/\pi]$ and $D_{(\pi D)}$ are
Noetherian normal domains by (I) and (III).

\smallskip

(2) Now suppose that $R$ is a UFD. Then $D[1/\pi]$ is a UFD by (I).
Since $D$ is a Krull domain, $\pi$ is a prime element of $D$
and $D[1/\pi]$ is a UFD, it follows that $D$ is a UFD by Nagata's criterion (\cite[Corollary 7.3]{F}). 

\smallskip

(3) Let $p$ be a nonzero prime element of $R$. 
If $pR= \pi R$, then $p D=\pi D$, so that
$p$ is prime in $D$ by (II).

So we consider the case $pR \ne \pi R$. In this case, $p\notin\pi R$, and hence
$p \notin \pi D$ because of (II), which implies that 
$pD_{(\pi D)}=D_{(\pi D)}$. 

Since $D= D[1/\pi] \cap D_{(\pi D)}$, it thus follows that
$pD= pD[1/\pi] \cap D$. Note that $p$ is a prime element of
$R[1/\pi]$ because $pR\ne \pi R$, and hence $p$ is prime in $R[1/\pi][X] = D[1/\pi]$.
Thus $pD[1/\pi]$ is a prime ideal of $D[1/\pi]$, which implies
that $pD$ ($=pD[1/\pi]\cap D$) is a prime ideal of $D$, as
desired. This completes the proof.
\end{proof}

\begin{lem}\label{l2}
Let $R$, $\pi$ and $D$ be as in Lemma \ref{l1} 
{\rm (}with conditions {\rm (I)}, {\rm(II)} and {\rm (III)}{\rm )}. 
Suppose that $R[X]\nsubseteq D$ and $D \nsubseteq R[X]$.
Set $A:= R[X] \cap D$, $P_1:= \pi R[X] \cap A$ and $P_2:= \pi D \cap A$. 
Then the following assertions hold.
\begin{enumerate}
\item [\rm(1)] $A$ is a Krull domain. Further,
$A = R[X] \cap D_{(\pi D)}$ and $A[1/\pi]=
R[1/\pi][X]$ ($=D[1/\pi]$).
\item[\rm (2)] $\pi A = P_1 \cap P_2$, $P_1 \nsubseteq P_2$, 
and $P_2 \nsubseteq P_1$.
\item [\rm(3)] $A_{P_1} = R[X]_{(\pi R[X])}$ and 
$A_{P_2} = D_{(\pi D)}$.
\item [\rm(4)] For each $f \in P_1$, $A [1/f] = D[1/f]$ and for each 
$f \in P_2$, $A[1/f] = R[X][1/f]$.
\item [\rm(5)] For any prime element $p$ in $R$ with $pR\ne \pi R$, 
$pA= pR[X]\cap A$
and hence $p$ remains a prime element in $A$.
\item [\rm (6)] If $(R,\m)$ is  a two-dimensional regular local ring with $\m=(\pi, t)R$, then $A$ is a faithfully flat $R$-algebra. 
\end{enumerate}
\end{lem}

\begin{proof}
(1) Since $D$ is a Krull domain by Lemma \ref{l1}(1) and $R$ is a Noetherian
normal domain, $A$ ($=R[X]\cap D$) is a Krull domain.

Since $D[1/\pi] = R[1/\pi][X]$,
we have $A[1/\pi]=R[1/\pi][X]\cap D[1/\pi]=R[1/\pi][X]$. Moreover, since
$D = D[1/\pi] \cap D_{(\pi D)}$ by Lemma \ref{easy} and 
$R[X]\subseteq D[1/\pi]$, we have
$$
A= R[X]\cap D= R[X] \cap (D[1/\pi] \cap D_{(\pi D)})=R[X]\cap D_{(\pi D)}.
$$

(2) Since $\pi A = \pi R[X] \cap \pi D$, we have 
$\pi A = P_1 \cap P_2$. 

We shall show that $P_1 \nsubseteq P_2$ by contradiction. Suppose, if possible,  that 
$P_1 \subseteq P_2$. Then $\pi A = P_1= \pi R[X] \cap A$, while
$A[1/\pi] = R[X][1/\pi]$ by (1). Since $A\subseteq R[X]$, it thus follows
from Lemma \ref{easy} that $A=R[X]$. This implies that
$R[X] \subseteq D$, which contradicts our hypothesis.

Similarly we have $P_2 \nsubseteq P_1$. 

\smallskip

(3) Since $A$ is a Krull domain by (1) and $\pi A = P_1 \cap P_2$ is the 
irredundant prime decomposition of $\pi A$ in $A$ by (2), it follows that
$A_{P_i}$ is a DVR for $i=1,2$. Therefore $A_{P_1} = R[X]_{(\pi R[X])}$ and
$A_{P_2} = D_{(\pi D)}$, because $A_{P_1}\subseteq R[X]_{(\pi R[X])}$, 
$A_{P_2}\subseteq D_{(\pi D)}$ and $A[1/\pi]=R[1/\pi][X]=D[1/\pi]$.

\smallskip

(4) Let $f \in P_2$. Clearly $A[1/f] \subseteq R[X][1/f]$ and by (2),  $\pi A[1/f] = P_1A[1/f] = 
\pi R[X][1/f] \cap A[1/f]$. Since $A[1/f][1/\pi] = R[X][1/f, 1/\pi]$ by (1), 
it follows from Lemma \ref{easy} that $A[1/f]=R[X][1/f]$.
Similarly, we have $A[1/f] = D[1/f]$ for each $f \in P_1$. 

\smallskip

(5) Since $pR\ne \pi R$, we have 
$pD_{(\pi D)}=D_{(\pi D)}$ by Lemma \ref{l1}(3). It then follows from (1)
that $pA=pR[X]\cap D_{(\pi D)}=pR[X] \cap A$. Thus $p$ is prime in $A$.

\smallskip

(6) By (5), $t$ is a prime element in $A$ and hence $\{t, \pi\}$ is a regular sequence in $A$. Therefore,
by Lemma \ref{cflat}, $A$ is a flat $R$-algebra.  Since $A \subseteq R[X]$, it follows that $A$ is faithfully flat over $R$.
\end{proof}

\begin{lem}\label{l3}
Let $R$, $D$, $A$, $P_1$ and $P_2$ be as in Lemmas \ref{l1} and \ref{l2}. 
Then the following statements hold.
\begin{enumerate}
 \item [\rm(1)] If $P_1+P_2=A$, then 
$A/\pi A \cong A/P_1 \times  A/P_2$,  $A/P_1 \cong (R/\pi R)[X]$ 
and  $A/P_2 \cong D/\pi D$.
 \item [\rm(2)] If $\pi D$ is a maximal ideal of $D$, 
then $D$ is a Noetherian domain.

\item[\rm (3)] If $D$ is a Noetherian domain, then $A [1/f]$ is 
Noetherian for every $f \in P_1 + P_2 $. In particular, 
if $P_1+P_2 =A$, then $A$ is a Noetherian domain.  

\item [\rm (4)] If $D$ is a finitely generated $R$-algebra, 
then $A [1/f]$ is a
finitely generated $R$-algebra for every $f \in P_1 + P_2 $. In particular, 
if $P_1+P_2 =A$ then $A$ is a finitely generated $R$-algebra.  
\end{enumerate}
\end{lem}

\begin{proof}
(1) Assume $P_1+P_2 =A$. Then,
since $\pi A= P_1 \cap P_2$ by Lemma \ref{l2}(2), we have
$A/\pi A \cong A/P_1 \times  A/P_2$.

We now show that $A/P_1\cong R[X]/\pi R[X]$, 
namely, $R[X]=A+\pi R[X]$. Let $f\in R[X]$.
Since $P_1 +P_2 =A$, there exist $a \in P_1$ and $b \in P_2$  
such that $a+b =1$. Then $A[b^{-1}] = R[X][b^{-1}]$
by Lemma \ref{l2}(4), so that
$b^nf=(1-a)^nf \in A$ for some $n>0$. Since $a\in P_1\subseteq \pi R[X]$,
from this it follows that $f\in A+\pi R[X]$, as desired. 

Similarly we have $A/P_2 \cong D/\pi D$.

\smallskip

(2) Since $D[1/\pi]$ ($= R[1/\pi][X]$) is Noetherian and $\hgt (\pi D)=1$
(as $D_{(\pi D)}$ is a DVR), the assertion follows from Lemma \ref{cohen}.

\smallskip

(3) Let $f\in P_1+P_2$, and write $f=a+b$ with $a\in P_1$ and
$b\in P_2$. Then $A[1/a]=D[1/a]$ and $A[1/b]=R[X][1/b]$ by
Lemma \ref{l2}(4), so that both $A[1/a]$ and $A[1/b]$ are
Noetherian. Since $(a,b)A[1/f]=A[1/f]$, it then follows from
Lemma \ref{noethl}(1) that $A[1/f]$ is Noetherian.

\smallskip

(4) The proof is similar to the above proof of (3).
\end{proof}

We now present our main example over the 
complete regular local domain 
$R=\bC[[u,v]]$, where $u$ and $v$ are indeterminates over $\bC$:
we construct a Noetherian normal $R$-subalgebra of $R[X]$ which is not 
finitely generated over $R$. This example shows that Theorem \ref{thmB} does not 
extend to complete local rings of dimension two and that the hypothesis on the transcendence degree
of certain fibres is necessary in Theorem \ref{Main}.

\begin{ex}\label{ex1}
Let $R=\bC[[u,v]]$, where $u$ and $v$ are indeterminates over $\bC$, and
let $p_n$ denote the $n$-th prime number for $n>0$.
We set 
$$
L := \bC((v))[v^{1/p_1}, v^{1/p_2}, \dots, 
v^{1/p_{n}}, \dots],
$$
so that $L$ is an infinite algebraic extension of the field $\bC((v))$.

Let $x_0=uX$, $x_1=(vx_0-1)/u$, and
\begin{eqnarray*}
x_{n} &=& \frac{{x_{n-1}}^{p_{n-1}}-v}{u}\\
\end{eqnarray*}
for $n>1$. 
Let $D=R[x_0,\, x_1, \dots, x_n,\,\dots]$ and $A= R[X]\cap D$. 
Then the following hold.
 \begin{enumerate}
\item [\rm(1)] $D[1/u]= R[1/u][X]$.
\item [\rm(2)] $D/uD\cong L$; in particular, $D/uD$ is not finitely generated over $R/uR$.
\item [\rm(3)] $uD$ is a maximal ideal of $D$, $uD \cap R= uR$ and $\hgt (uD)=1$.
\item [\rm(4)] $D$ is a Noetherian UFD which is not finitely 
generated over $R$.
\item [\rm(5)]$A$ is a Noetherian normal domain which is faithfully flat but not 
finitely generated over $R$.
\end{enumerate}
\end{ex}

\begin{proof}
(1) Since $x_0= uX$ and $x_i \in R[1/u][X]$ for each $i > 0$,
we have $D[1/u]= R[1/u][X]$.

\smallskip

(2) Set 
$$
f_1:=uX_1-vX_0+1 {\rm ~~and~~} f_n:=uX_n-{X_{n-1}}^{p_{n-1}}+v {\rm ~~for~} n\ge 2,
$$ 
where $X_0,X_1,\dots,X_n,\dots$ are indeterminates over $R$. Let 
$$
I_n=(f_1,f_2,\dots,f_n)R[X_0,X_1,\dots,X_n] {\rm ~~for~} n\ge 1
$$
and let $I_0$ denote the zero ideal of $R[X_0]$. Set
$$
C_0:=R[X_0] {\rm ~~and~~} C_n:= R[X_0, X_1, \cdots, X_n]/I_n {\rm ~~for~} n\ge 1.
$$
We shall prove, by induction on $n$, that for each $n\ge 1$, $I_n \cap R[X_0,X_1,\dots,X_{n-1}] =I_{n-1}$ 
and $I_n$ is a prime ideal of $R[X_0, X_1, \ldots , X_n]$ with $I_n \cap R=(0)$. 
This would establish that each $C_n$ is an integral domain and that we may identify   
$C_{n-1}$ with its canonical image in $C_n$, i.e.,  we may assume that, for every $n \ge 1$,
\[
R\subseteq C_0\subseteq\cdots\subseteq C_{n-1}\subseteq C_n.
\]
Note that $f_1$ is an irreducible polynomial in $R[X_0,X_1]$ which is a UFD so that 
$I_1$ is a prime ideal of $R[X_0, X_1]$. Further $I_1 \cap C_0=(0)$ and $I_1 \cap R= (0)$. 
Thus the assertion holds for $n=1$.    

Now suppose that the assertion holds for $n$. Let
$J=I_{n+1}\cap R[X_0,\dots,X_n]$. We first show that $J=I_n$. Clearly, $I_n \subseteq J$.
 Note that
\[
R[u^{-1}][X_0,X_1,\dots,X_m]=R[u^{-1}][X_0,f_1,\dots,f_m]
\]
for every $m$, which implies that 
$J[u^{-1}]=I_n[u^{-1}]$.  Now let $h\in J$. Then
$u^rh\in I_n$ for some $r>0$. By induction hypothesis,  $I_n$ is a prime ideal of 
$R[X_0, X_1, \cdots, X_n]$  with $I_n\cap R=(0)$.
In particular, $u\notin I_n$ as $u \in R$, so that $h\in I_n$. Thus we have
$J=I_n$, as claimed. Therefore, we may assume that $C_n\subseteq C_{n+1}$.
Note that, $I_{n+1} \cap R= I_{n+1} \cap R[X_0,\dots,X_n] \cap R=I_n \cap R=(0)$. 

We next show that $I_{n+1}$ is a prime ideal of $R[X_0, X_1, \cdots, X_{n+1}]$, i.e., 
$C_{n+1}$ is an integral domain.
Let $z_n$ denote the image of $X_n$ in $C_n$, set $w_n:={z_{n}}^{p_{n}}-v \in C_n$ and  
$$
g_{n+1}:=uX_{n+1}-w_n \in C_n[X_{n+1}], 
$$ 
i.e., $g_{n+1}$ is the image of $f_{n+1}$ in $C_n[X_{n+1}]$.
Since $C_n\subseteq C_{n+1}$, it then follows that 
$$
C_{n+1}=C_n[z_{n+1}]\cong C_n[X_{n+1}]/(g_{n+1}) {\rm~~and~~} C_{n+1}[1/u] \cong C_n[1/u].
$$ 
Also setting $L_{n}:=\bC((v))[v^{1/2}, v^{1/3}, \dots, v^{1/p_{n}}]$, an algebraic extension of $\bC((v))$, 
we  have
$$C_n/uC_n\cong\bC[[v]][X_0,X_1,\dots,X_n]/(vX_0-1,{X_1}^2-v,\dots,
{X_{n-1}}^{p_{n-1}}-v)\cong L_{n-1}[X_n], $$
and hence $C_n/uC_n$ is an integral domain.
This implies that 
$u, w_n$ is a regular sequence in $C_n$ and hence a regular sequence in $C_n[X_{n+1}]$.
Thus, $u$ is a regular element of $C_{n+1}$, so that the canonical map 
$C_{n+1} \to C_{n+1}[1/u]$ is injective.  But $C_{n+1}[1/u]$ ($\cong C_n[1/u]$) is an integral domain.   
Therefore,
$C_{n+1}$ is an integral domain,  as desired.   

Since $w_n$ is the image of ${X_n}^{p_n}-v$ in $C_n$, we have
$$C_n/(u,w_n)C_n\cong\bC[[v]][X_0,X_1,\dots,X_n]/(vX_0-1,{X_1}^2-v,\dots,
{X_{n-1}}^{p_{n-1}}-v, {X_n}^{p_n}-v) \cong L_{n}.$$ 
Thus, we have a canonical isomorphism 
\begin{equation}\label{cl}
{\psi}_n \colon C_n/(u,w_n)C_n \stackrel{\simeq}{\to} L_n.
\end{equation}
and hence a canonical surjection 
$\theta_n: C_n \to L_n$.

Now, let $C= \bigcup_{n \ge 0}C_n$. As $C_0=R^{[1]}$ and $C_{n+1}$ is algebraic over $C_n$ $\forall$ $n \ge 0$,
we have ${\rm tr.deg}_R \,C =1$.
For each $n \ge 0$, let 
$$
\Phi_n\colon R[X_0,X_1,\dots,X_n]\to D
$$ 
be the $R$-algebra map defined by 
$$\Phi_n(X_i)=x_i
 {\rm ~~for~~} 0\le i\le n.$$ Then, for each $j$, $1 \le j \le n$,  
$\Phi_n(f_j)=ux_j-{x_{j-1}}^{p_{j-1}}+v=0$, so that $\Phi_n$ induces an $R$-algebra map
$$
{\phi}_n\colon C_n\to D
$$ 
such that 
$$
{\phi}_n(z_i)=x_i {\rm ~~ for ~~} 0 \le i \le n.
$$
In particular, for any $m \ge 0$,
we have  
\begin{equation}\label{m}
{\phi}_m(z_m)=x_m.
\end{equation} 
Since ${\rm tr.deg}_R \,C_n =1 = {\rm tr.deg}_R \, D$, and
since both $C_n$ and $D$ are integral domains, it then follows that
${\phi_n}$ is an injective $R$-algebra homomorphism. We thus have an injective $R$-algebra map 
$$
\phi: C \to D {\rm ~~such~~that~~} \phi|_{C_n}={\phi}_n,
$$
which is also surjective as, for each $m \ge 0$, $\phi (z_m)={\phi}_m(z_m)=x_m$ by (\ref{m}). Thus, $C \cong D$.

Note that 
$$ 
{\phi}_n (w_n)= {x_n}^{p_n}-v=ux_{n+1} \in uD,
$$
so that ${\phi}_n (u,w_n)C_n \subseteq uD$. Hence, as $L_n$ is a field, by the isomorphism ${\psi}_n$ in (\ref{cl}), ${\phi}_n$ induces 
an injective map 
$$
\overline{{\phi}}_n\colon L_n\to \dfrac{D}{uD}.
$$ 
Note that, for any $m \ge 0$, 
\begin{equation}\label{surjection} 
\overline{{\phi}_m} (\theta_m({z_m}))=\overline{x_m}, 
\end{equation}
where $\overline{x_m}$ denotes the image of $x_m$ in $D/uD$.
Now, the $\overline{\phi_n}$'s give rise to an injective $R$-algebra map 
$$
\overline{{\phi}}\colon L\to \dfrac{D}{uD},
$$
which is also surjective as, given any $m \ge 0$, 
$\overline{{\phi}} (\theta_m({z_m}))=\overline{x_m}$ by (\ref{surjection}).  
Thus $\overline{\phi}$ is an isomorphism, i.e., $D/uD \cong L$.

As $L$ ($\cong D/uD$) is algebraic over $\bC[[v]]$ ($\cong R/uR$) and hence over $\bC((v))$, and as
$L$ is not a finite extension of $\bC((v))$, it follows that $D/uD$ is not a finitely generated  algebra
over $R/uR$.

\smallskip

(3) $uD$ is a maximal ideal by (2). 

Since $\bC[[v]] \subseteq L$, the canonical map from $R/uR$ ($=\bC[[v]]$) to 
$D/uD$ ($=L$ by (2)) is injective. Hence $uD \cap R =uR$. 

We show that $\hgt (uD)=1$. 
Since $R$ is a Noetherian domain and $uD \cap R =uR$, it follows from Theorem \ref{C} that 
$$
\hgt (uD) + {\rm tr.deg}_{R/uR}\, (D/uD)  \leq 
\hgt (uR) + {\rm tr.deg}_R \, D = 2 ,
$$
so that $\hgt (uD) \le 2.$ Now ${\rm tr.deg}_{R/uR}\, (D/uD)=0$ by (2).
Therefore, if $\hgt (uD) =2$, then $uD$ would satisfy the 
dimension formula relative to $R$ and hence,  by \cite[Theorem 3.6]{O}, it would follow that 
$D/uD$ is a subalgebra of a finitely generated $(R/uR)$-algebra. But then, as $D/uD$ is a field by (2),
it would follow from Proposition \ref{giral} that $D/uD$ itself is a finitely generated $R/uR$-algebra,
which is not the case by (2).
Hence, $\hgt (uD) = 1$. 

\smallskip

(4) From (1) and (3), it follows, by Lemma \ref{cohen}, that $D$ is Noetherian.
$D$ is a UFD by Lemma \ref{l1}(2).
Since $D/uD$ is not finitely generated over $R/uR$ by (2), 
it follows that $D$ is not finitely generated over $R$.

\smallskip

(5) $A$ is faithfully flat over $R$ by Lemma \ref{l2}(6).

$A$ is a Krull domain by Lemma \ref{l2}(1), so that $A$ is normal. 

We show that $A$ is Noetherian.
Let $P_1 = u R[X] \cap A$ and $P_2 = uD\cap A$ (as in Lemma \ref{l2}).
Since $x_0 = uX \in uR[X] \cap A = P_1$ and $vx_0-1 (=ux_1) \in uD \cap A = P_2$, 
we have $P_1 + P_2 =A$. Thus $A$ is Noetherian by Lemma \ref{l3}(3).

Since $D/uD$ is not finitely generated over $R/uR$ by (2) and
$A/P_2 \cong D/uD$ by Lemma \ref{l3}(1), it follows that 
$A/P_2$ is not finitely generated over $R/uR$ and hence 
$A$ is not finitely generated over $R$.
\end{proof}

The following example shows that the condition in Lemma \ref{l3} that 
$D/\pi D$
is a field  cannot be replaced by the condition that $D/\pi D$ 
is Noetherian in order to conclude that $D$ is Noetherian.  
It also illustrates the necessity of the hypothesis that $R$ is complete
in Proposition \ref{lNoeth}.

\begin{ex}\label{non-noeth}
Let $R=k[u,v]_{(u,v)}$, where $k=\bar{\bQ}$ is 
the algebraic closure of $\bQ$ 
and $u, v$ are indeterminates over $k$. Set
$a_0:=1$ and $a_n:=v^n/n!$ for $n >0$. 
Set $x_0 := uX$, where $X$ is an indeterminate over $R$ and 
\begin{eqnarray*}
x_n: &=& \frac{x_{n-1}-a_{n-1}}{u}=
\frac{x_0-a_0-a_1u-\cdots-a_{n-1}u^{n-1}}{u^{n}}
\end{eqnarray*}
for $n \ge 1$. 
Let $D=R[x_0,x_1, \dots, x_n, \dots]$. 
Then the following statements hold.
\begin{enumerate}
\item [\rm(1)] $D[1/u]=R[1/u][X]$.
\item [\rm(2)] $u$ is a prime element in $D$ and $uD \cap R=uR$.
\item [\rm(3)] $D/uD=R/uR$. In particular, $D/uD$ is Noetherian and 
$D/(u,v)D$ is a domain.
\item [\rm(4)] $D_{(uD)}$ is a discrete valuation ring.
\item [\rm(5)] $D$ is a non-Noetherian UFD.
\end{enumerate}
\end{ex}
\begin{proof}
We first note that $D=\bigcup_{n\ge0}R[x_n]$ and for each $n \ge 0$,
$x_n=ux_{n+1} + a_n$ for some $a_n \in R$. The assertions (1)--(3)
follow from this.

\smallskip

(4) Set $\widetilde{R}:= k[v][[u]]_{(u,v)}$. Then  $\widetilde{R}$ is a Noetherian ring
being the $u$-adic completion of $R$. We show that
$D$ is $R$-isomorphic to a subring $D'$ of $\widetilde{R}$.
We set
\[
y:=a_0+a_1u+a_2u^2+\cdots+a_nu^n+\cdots,
\]
so that $y$ is an element of $k[v][[u]]_{(u,v)} (= \widetilde{R})$. 
Since $a_0=1$ and $a_n=v^n/n!$ for each $n >0$, it follows that
\[
y=1+\frac{vu}{1}+\frac{{(vu)}^2}{2!}+\cdots+\frac{{(vu)}^n}{n!}+\cdots.
\] 
Thus $y$ is a transcendental element over $R$. We now define 
a subring $D'$ of $\widetilde{R}$ by 
\[
D':=R[y,\frac{y-a_0}{u}, \frac{y-a_0-a_1u}{u^2}, \dots, 
\frac{y-a_0-a_1u-\cdots-a_{n-1}u^{n-1}}{u^{n}}, \dots]. 
\]
Let $\phi : D \to D'$ be the $R$-linear map defined 
by $\phi(x_0)= y$. Since $y$ is transcendental over $R$, it follows 
that $\phi$ is an isomorphism. 

Note that we have $\bigcap_{n\ge 0}u^nD'\subseteq 
\bigcap_{n\ge 0}u^n\widetilde{R} =(0)$. Since $\phi$ is an isomorphism, from this
it follows that $\bigcap_{n\ge 0}u^nD=0$, which implies that $\hgt (uD)=1$. 
As $D$ is a Krull domain by Lemma \ref{l1}(1), it follows that $D_{(uD)}$ is a DVR.

\smallskip

(5) $D$ is a UFD by Lemma \ref{l1}(2).
Note that $v\in R$ remains a prime element in $D$ by Lemma \ref{l1}(3).
Since
\[
x_0-1=x_0-a_0=u^nx_n+a_1u+\cdots+a_{n-1}u^{n-1}
\]
and $a_i\in vR\subseteq vD$ for $i>0$, we have
$x_0-1 \in \bigcap_{n \ge 0}u^n(D/vD)$, which shows that
the integral domain $D/vD$ is not Noetherian. Thus $D$ is not
Noetherian. 
\end{proof}

Let $R$ be a Noetherian normal domain and $A$ be a Krull domain such that
$R \subseteq A \subseteq R[X]$. If $\dim R = 1$, then
$A$ is a Noetherian domain by \cite[Lemma 3.3]{DO2}. The following example 
shows that the result does not hold when $\dim R= 2$, not even if 
$R$ is a complete regular local ring.

\begin{ex}\label{ex2}
Let $R=\bC[[u,v]]$, where $u$ and $v$ are indeterminates over $\bC$. 
Let $p_n$ denote the $n^{\rm th}$ prime number and set 
$q_n=\Pi_{i=1}^n p_i$ for $n>0$.
Let $S$ be the infinite integral extension of $\bC[[v]]$ generated 
by all the $v^{1/q_n}$'s, i.e.,
$$S= \bigcup_{n \ge 1}\bC[[v^{1/q_n}]].$$
Set $x_0:= uX$, $x_1:=({x_0}^2-v)/u$ and
\begin{eqnarray*}
x_n &=& \frac{{x_{n-1}}^{p_{n}}-x_{n-2}}{u}
\end{eqnarray*}
for $n\ge2$. Let $D=R[x_0,\, x_1, \dots, x_n,\,\dots]$ and $A=D \cap R[X]$.
Then the following hold.
\begin{enumerate}
\item [\rm(1)] $D[1/u]=R[1/u][X]$. 
\item [\rm(2)] $uD$ is a prime ideal of $D$ and $\hgt (uD)=1$.
\item [\rm(3)] $D/uD\cong S$, and hence $D/uD$ is a non-Noetherian ring.
\item [\rm(4)] $D$ is a non-Noetherian UFD.
\item [\rm(5)] $A$ is a non-Noetherian Krull subalgebra of $R[X]$.
\end{enumerate}
\end{ex}
\begin{proof}
Proof of (1) follows easily.
Following the same argument as in Example \ref{ex1}, one can show that
$D/uD \cong C[[v]][X_0, X_1,X_2, \dots]/({X_0}^2-v,{X_1}^{3}-X_0, {X_{2}}^{5}-X_1, \dots)\cong S$
and $\hgt (uD)=1$.
Note that $S= \bigcup_{n \ge 1}\bC[[v^{1/q_n}]]$ is 
a direct limit of discrete valuation rings $\bC[[v^{1/q_n}]]$
and hence a valuation ring. It is an infinite integral extension over $R/uR = \bC [[v]]$.
Further $S$ has following infinite ascending chain of ideals
$$
(v)\subsetneqq (v^{1/2})\subsetneqq\cdots \subsetneqq (v^{1/q_n})\subsetneqq \cdots ,
$$
showing that $D/uD$ is not Noetherian.
Thus (2) and (3) follow from above.

\smallskip

(4) $D$ is a UFD by Lemma \ref{l1}. Since $D/uD$ is not Noetherian,
$D$ is {\it not Noetherian}.

\smallskip

(5) $A$ is Krull domain by Lemma \ref{l2} (1).
Note that $D$ and $A$ are of the ``types'' as in Lemmas \ref{l1} 
and \ref{l2}. Recall that, $P_1 = u R[X] \cap A$ and $P_2 = uD\cap A$.
Now $x_0 = uX \in uR[X] \cap A = P_1$ and 
${x_0}^2-v \in uD \cap A = P_2$ and so
$v \in P_1 + P_2$. Thus, by Lemma \ref{l3} (1),
$(A/P_2)[1/v] = (D/uD)[1/v]$. 
Therefore, $A/P_2$ and $D/uD$ are birational.
Therefore, since $A/P_2$ is a one dimensional domain
and $D/uD$ is not a Noetherian ring, 
$A/P_2$ is not Noetherian by 
Krull-Akizuki theorem (cf. \cite[Theorem 11.7]{M}). 
Thus, $A$ is not a Noetherian ring. 
\end{proof}

\section{Appendix}

The following result on finite generation of algebras has been proved 
for the case $R$ is a field in \cite[Theorem 1.1]{Ot}.
Below, we show that the result can be extended to an excellent 
domain. The proof is essentially the same as in \cite{Ot}. 

\begin{prop}\label{toya}
Let $R$ be an excellent local domain and $A$ a normal domain containing $R$. 
 Suppose that there exists a nonzero element $f$ in $A$ 
satisfying the following hypotheses. 
\begin{enumerate}
 \item [\rm (I)] $A[f^{-1}]$ is a finitely generated $R$-algebra.
\item [\rm (II)] $A/fA$ is a finitely generated $R$-algebra.
\item [\rm (III)] For every minimal prime ideal $P$ of $f$ in $A$, 
$\hgt P =1$ and $P$ satisfies the dimension formula relative to $R$.
\end{enumerate}
Then $A$ is finitely generated over $R$.
\end{prop}

\begin{proof}
Since $A[f^{-1}]$ is a finitely generated $R$-algebra, by Theorem \ref{onoda}(1),
it is enough to show that $A_M$ is a locality (essentially of finite type) over $R$
for every maximal $M$ of $A$.
Let $M$ be a maximal ideal of $A$. Since $A[f^{-1}]$ is finitely generated 
over $R$, we may assume that $f \in M$ for our consideration.

Let $P$ be a minimal prime ideal of $f$ in $A$ and $\p= P \cap R$.
Since $P$ satisfies the dimension formula relative to $R$, we have
\begin{equation}\label{e1}
\hgt P + {\rm tr.deg}_{R/\p} \, A/P = \hgt \p + {\rm tr.deg}_R \, A. 
\end{equation}
Since $R$ is an excellent ring, it is universally catenary. Hence, since, by (II),
$A/P$ is a finitely generated ($R/\p$)-algebra, we have
\begin{equation}\label{e2}
\hgt (M/P) + {\rm tr.deg}_{R/(M\cap R)} \, A/M = \hgt ((M \cap R)/\p) + {\rm tr.deg}_{R/\p} \,A/P
\end{equation}
by \cite[Theorem 15.6]{M}.
From (\ref{e1}) and (\ref{e2}), and the fact that $R$ is catenary,  we have
\begin{equation}\label{e3}
 \hgt P + \hgt (M/P) = \hgt (M \cap R) +  {\rm tr.deg}_R \, A - {\rm tr.deg}_{R/(M\cap R)} \, A/M . 
\end{equation}
By Theorem \ref{C}, we have 
\[
\hgt M \le  \hgt (M \cap R) +  {\rm tr.deg}_R \, A  - {\rm tr.deg}_{R/(M\cap R)} \, A/M.
\]
Hence by (\ref{e3}), we have
\[
\hgt P + \hgt (M/P) \ge \hgt M. 
\]
Thus, for any minimal prime ideal $P$ of $f$ 
in $A$, 
\begin{equation}\label{e4}
\hgt M = \hgt P + \hgt (M/P)= \hgt (M/P) +1, 
\end{equation}
since $\hgt P =1$.

Let $\widehat{A}$ denote the $M$-adic completion of $A_M$. 
We show that $\widehat{A}$ is a Noetherian ring 
such that $\dim \widehat{A} \ge \hgt M$.
Let $\widetilde{A}$ be the $f$-adic completion of $A_M$. 
Then $\widetilde{A}/f\widetilde{A} \cong A_M/fA_M$, which is Noetherian by hypothesis (II).
Hence $\widetilde{A}$ is Noetherian (cf. \cite[Corollary 4, p.260]{ZS}).
Since $\widehat{A}$ coincides with the $M$-adic completion of $\widetilde{A}$,
we  have $\widehat{A}$ is a Noetherian ring and 
$$
\dim \widehat{A} = \dim \widetilde{A}.
$$
On the other hand $f$ is a regular element of $\widetilde{A}$. Hence, by (\ref{e4}), 
\begin{equation}\label{e5}
\dim \widehat{A} = \dim \widetilde{A} \ge \dim \widetilde{A}/f\widetilde{A} +1 = \dim A_M/fA_M +1 = \hgt M. 
\end{equation}

Since $A[f^{-1}]$ and $A/fA$ are finitely generated $R$-algebras, there exists a 
finitely generated $R$-algebra $C$ such that 
$C[f^{-1}] = A[f^{-1}]$ and $C/ (fA \cap C) = A/fA$.
Since $f \in M$, $M$ is a finitely generated ideal of $A$ by hypothesis (II) and we may assume that
the generators of $M$ are contained in $C$, so that $M= (M \cap C)A$.
Let $B$ be the integral closure of $C$ in its field of fractions. Since $R$ is excellent
and $C$ is a finitely generated $R$-algebra, $B$ is also finitely generated over $R$.
Since $A$ is normal and birational to $C$, $B \hookrightarrow A$ and we have

\begin{enumerate}
\item[{\rm(1)}] $B[f^{-1}] = A[f^{-1}]$.
\item[{\rm(2)}] $B / (fA \cap B)= A/fA$.
\item[{\rm(3)}] $B/ \m = A/ M$ and $\m A= M$, where $\m= M \cap B $.
\item[{\rm(4)}] $B$ is a  finitely generated normal $R$-algebra.
\end{enumerate}

Let $\widehat{B}$ be the $\m$-adic completion of $B_\m$. 
Since $R$ is an excellent domain and $B$ is a finitely generated normal $R$-algebra, we have
$\widehat{B}$ is a Noetherian normal domain 
(\cite[Theorem 79, p. 258]{Ma}).
Let $\phi : \widehat{B} \to \widehat{A}$ be the canonical map induced by the inclusion
$B_\m \hookrightarrow A_M$. We show that $\phi$ is an isomorphism.

Since  $ \widehat{B}/ \m\widehat{B} = B/ \m = A/ M = \widehat{A}/M\widehat{A}$ and
$\widehat{A}$ is complete, we have $\phi$ is surjective (cf. \cite[Corollary 2, p.259]{ZS}). 
Now, since $\widehat{B}$ is an integral domain, to show that $\phi$ is injective, 
it is enough to show that $\dim \widehat{B} \le \dim \widehat{A}$.

Since $B$ is a finitely generated $R$-algebra, we have
$$\dim \widehat{B} = \hgt {\m} = \hgt (\m / Q) + \hgt {Q},$$ where 
$Q = P\cap B$ and $P$ is a minimal prime ideal of $f$ in $A$.
By (2) and (3), we have $$\hgt (\m / Q) = \hgt (M/P).$$
Using (\ref{e1}), and the facts that $Q= P\cap B$, 
$B$ is birational to $A$ such that $B/Q= A/P$ and 
$Q$ satisfies the dimension formula relative to $R$, we have $$\hgt Q= \hgt P.$$
Hence, by (\ref{e4}) and (\ref{e5}), we have
\[
\dim \widehat{B} = \hgt {\m} = \hgt (M/P) + \hgt P = \hgt M \le \dim \widehat{A}.
\]
Thus, $\phi$ is an isomorphism and we may identify $\widehat{B}$ with $\widehat{A}$. 

Since $A$ is birational to $B$, $B_\m \subseteq A_M$ and
$\widehat{B} = \widehat{A}$, we have $A_M \hookrightarrow \widehat{A}$.

Now, since $\widehat{B}$ is faithfully flat over $B_{\m}$, we have
$B_\m = \widehat{B} \cap {\rm qt}(B)$, where ${\rm qt}(B)$ denotes the field of fractions of $B$.
Therefore, since $A_M  \subseteq  \widehat{A} \cap {\rm qt}(A)$,
$\widehat{B} = \widehat{A}$ and ${\rm qt}(B) = {\rm qt}(A)$, we have $A_M=B_{\m}$. 
Therefore, $A_M$ is a locality over $R$, as desired. This completes the proof. 
\end{proof}

{\bf Acknowledgements.} 
The second author acknowledges Department of Science and Technology for their SwarnaJayanti Fellowship.

\end{document}